\documentclass{amsart}
\usepackage{latexsym}
\usepackage{amssymb}
\newtheorem{lemma}{Lemma}[section]

\newtheorem{corollary}[lemma]{Corollary}
\newtheorem{theorem}[lemma]{Theorem}
\theoremstyle{definition}
\newtheorem{definition}[lemma]{Definition}
\theoremstyle{definition}

\newtheorem{remark}[lemma]{Remark}
\theoremstyle{definition}

\title {Observables on  synaptic algebras }
\author{Jen\v cov\'a, A., Pulmannov\'a, S.}
\address{Mathematical Institute, Slovak Academy of Sciences, \v Stef\'anikova 49, 814 73 Bratislava, Slovakia}
\email{jenca@mat.savba.sk, pulmann@mat.savba.sk}
\subjclass{06F25, 81B15, 08A72}
\keywords{Synaptic algebras; observables; smearings; effect algebras; GH-algebras}
\thanks{Supported by the grant VEGA 02/0069/16 and  by the grant of the Slovak Research and Development Agency under contract APVV-16-0073}
\begin{document}

\maketitle

\begin{abstract}  Synaptic algebras, introduced by D. Foulis, generalize different algebraic structures used so far as mathematical models of quantum mechanics: the traditional Hilbert space approach, order unit spaces, Jordan algebras, effect algebras, MV-algebras,  orthomodular lattices.  We study sharp and fuzzy observables on two special classes of synaptic algebras: on  the so called generalized Hermitian  algebras and on synaptic algebras which are Banach space duals. Relations between fuzzy and sharp observables on these two types of synaptic algebras are shown.
\end{abstract}

\section{Introduction}

In the traditional Hilbert space approach to quantum mechanics, as a proper mathematical  model of a physical quantity (so called  observable)  POV-measures (positive operator valued measures) are considered, instead of the more traditional  PV-measures  (projection valued measures). This approach provides also a frame to investigate imprecise measurements. The notion of fuzzy (or unsharp) observable has been formulated  in the literature (\cite{HLY} as a smearing of a sharp observable (PV-measure) by means of a (weak) Markov kernel. While in the classical mechanics, an unsharp observable is always a smearing of a sharp one, in quantum mechanics the situation is different.

PV-measures have ranges in the orthomodular lattice of projection operators, and are in one-to-one correspondence with self-adjoint operators. There is a well-developed functional calculus for commuting self-adjoint operators. It turns out that functions of one self-adjoint operator  may be considered as a special kind of smearings.
POV-measures have ranges in the algebra of Hilbert space effects (self-adjoint operators between the zero and identity operator). A special subclass of POV-measures are those with commuting ranges, which correspond exactly to smearings of some PV-measures \cite{ACHT, Hol, JPV1, JPchar}. Some of these results have been generalized in the literature for effect algebras or for their special subclasses, MV-algebras and orthomodular lattices \cite{JPV1, JPV2, Pu, Var}.

In this paper, we study observables in a more general frame of synaptic algebras.
Synaptic algebras were introduced by D. Foulis \cite{Fsyn} as  possible models for quantum mechanics.
The aim was to build a mathematical description of quantum theory based on a few relatively simple and physically relevant axioms.

Synaptic algebras put together several algebraic structures used so far  as models of quantum systems: the traditional Hilbert space model, Jordan algebras, order unit spaces, effect algebras, orthomodular lattices. In this paper, we show how some results obtained in the latter models can be adopted and enhanced for synaptic algebras.

In \cite{JPV1}, the notion of a weak Markov kernel was introduced and smearings of observables on effect algebras were studied. It turns out that on an arbitrary  $\sigma$-orthocomplete effect algebra, a smearing need not always exist. In section 3.1 of the present paper, we prove that on a convex $\sigma$-orthocomplete effect algebra with an ordering set of $\sigma$-additive states, there exists a unique smearing of every observable with respect to any weak Markov kernel.

As observables are usually defined as $\sigma$-homomorphisms from a $\sigma$-field of sets to a given algebraic structure, we need a kind of an additional $\sigma$-property on synaptic algebras.
Therefore in section 4, we study sharp observables and  unsharp observables with commuting ranges on a special kind of synaptic algebras, namely on GH-algebras (generalized Hermitian algebras), in which the set of projections is a $\sigma$-OML and every commutative sub-GH-algebra is monotone $\sigma$-complete.
Using a version of Loomis-Sikorski theorem for commutative GH-algebras, it was proved in \cite{FJPls}, that to every element of any GH-algebra there exists a unique sharp real observable. In this paper we show that also conversely,  every sharp real observable on a GH-algebra  $A$ is determined by an element of $A$. Moreover, a functional calculus for several commuting  sharp real observables is defined.
It is also shown that every observable with  commuting range is defined by a smearing of a  special sharp observable, whose existence follows by the Loomis-Sikorski representation. This result is similar to that  for POV-measures with commuting ranges \cite{JPV1}.

In the last section, we consider synaptic algebras which are duals of  Banach spaces. By \cite{ASS} and \cite{Sh}, this happens if and only if the synaptic algebra is isomorphic to a JW-algebra, hence a weakly closed Jordan algebra  of Hilbert space operators \cite{T}. Using the results of \cite{Sh} and \cite{Ellis}, we collect there some criteria under which a synaptic algebra is a Banach space dual, hence a JW-algebra.
We prove that in this case, for every observable  $\xi$ and every weak Markov kernel $\nu$ there is a unique observable $\eta$ defined by a smearing of $\xi$ by $\nu$. We also prove that an observable is a smearing of a sharp observable if and only if it has a commuting range. In case that the Hilbert space is separable, the sharp observable may be chosen real, which is a property shared with the POV-measures on a separable Hilbert space.

\section{Synaptic algebras}

In what follows, $A$ is a synaptic algebra with enveloping algebra $R
\supseteq A$, \cite{Fsyn, FPproj, FJPvl, TDSA, SymSA, ComSA, Punote, Pmodel} and $P$
is the orthomodular lattice \cite{Ber, Kalm} of projections in $A$.
A prototype example  is the set ${\mathcal A}$
of all self-adjoint operators in the algebra ${\mathcal B}({\mathcal H})$  of all bounded
linear operators on the Hilbert space ${\mathcal H}$  with
${\mathcal B}({\mathcal H})$ as enveloping algebra. See the literature cited above for numerous additional examples of synaptic algebras.

For the definition and axioms of a synaptic algebra see \cite{Fsyn}.
In this section we outline some of the notions and facts pertaining to
the synaptic algebra $A$. More details and
proofs can be found in \cite{Fsyn, FPproj, TDSA, SymSA, ComSA, Punote}.

If $a,b\in A$, then the product $ab$, calculated in the enveloping
algebra $R$, may or may not belong to $A$. However, if $ab=ba$, i.e.,
if $a$ commutes with $b$ (in symbols $aCb$), then $ab\in A$. Also, if
$ab=0$, then $aCb$ and $ba=0$. For $a\in A$, the set $C(a):=\{ b\in A: bCa\}$ is the \emph{commutant} of $a$ in $A$. If $B\subseteq A$, then $C(B):=\bigcap_{b\in B}C(b)$
is the commutant of $B$ in $A$,  $CC(B):=C(C(B))$ is the \emph{bicommutant} of $B$, and $CC(a):=CC(\{ a\})$.

The synaptic algebra $A$ is a partially ordered real linear space under
the partial order relation $\leq$ and we have $0<1$ (i.e., $0\leq 1$
and $0\not=1$); moreover, $1$ is an order unit in $A$, that is, for each $a\in A$ there exists $n\in {\mathbb N}$ such that $a\leq n1$. Moreover,  $A$ is Archimedean, i.e., if $a,b\in A$ and $na\leq b$ for all $n\in {\mathbb N}$, then $a\leq 0$. Thus $A$ is an \emph{order unit space} with the order unit norm $\|a\|:=\{ 0 < \lambda \in {\mathbb R}: -\lambda \leq a\leq \lambda\}$ \cite{Alf}. For every $a\in A$ the commutant $C(a)$ is a norm-closed subset of $A$.

Elements of the "unit interval" $E:=A[0,1]=\{e\in A:0\leq e\leq 1\}$ are called \emph{effects}, and $E$ is a \emph{convex effect algebra}
\cite{GP, GPBB, BGP}.

Elements of the set $P:=\{p\in A:p=p\sp{2}\}$ are called \emph{projections}
and it is understood that $P$ is partially ordered by the restriction of
$\leq$. The set $P$ is a subset of the convex set $E$ of effects in $A$; in
fact, $P$ is the extreme boundary of $E$ (\cite[Theorem 2.6]{Fsyn}).
Evidently, $0,1\in P$ and $0\leq p\leq 1$ for all $p\in P$. It turns out
that $P$ is a lattice, i.e., for all $p,q\in P$, the \emph{meet} (greatest
lower bound) $p\wedge q$ and the \emph{join} (least upper bound) $p\vee q$
of $p$ and $q$ exist in $P$; moreover, $p\leq q$ iff $pq=qp=p$.
The projections $p$ and $q$ are said to be \emph{orthogonal}, in symbols
$p\perp q$, iff $p\leq 1-q$. The \emph{orthosum} $p\oplus q$ is
defined iff $p\perp q$, in which case $p\oplus q:=p+q$. It turns out that
$p\perp q\Leftrightarrow pCq$ with $pq=qp=0$; furthermore, $p\perp q
\Rightarrow pCq$ with $p\oplus q=p+q=p\vee q\in P$. The lattice $P$,
equipped with the orthocomplementation $p\mapsto p^{\perp}=1-p$, is an
 \emph{orthomodular lattice} (OML) \cite{Ber, Kalm}.

Two elements $e,f\in E$ are called \emph{compatible} if there are elements $e_1,f_1$ and $g$ such that $e=e_1+g$, $f=f_1+g$, and $e_1+f_1+g\in E$. If $p, q\in P$, then $p$ and $q$ are compatible in $P$ (that is, there are $p_1,q_1,r\in P$ with $p_1+q_1+r\in P$, such that $p=p_1+r, q=q_1+r$) iff they are compatible in $E$. Moreover,  $p$ and $q$ are compatible iff $pCq$ \cite{Fsyn}.

If $0\leq a\in A$, then there is a uniquely determined element $r\in A$
such that $0\leq r$ and $r\sp{2}=a$; moreover, $r\in CC(a)$ \cite[Theorem 2.2]
{Fsyn}. Naturally, we refer to $r$ as the \emph{square root} of $a$, in
symbols, $a^{1/2}:=r$. If $b\in A$, then $0\leq b^{2}$, and the \emph
{absolute value} of $b$ is defined and denoted by $|b|:=(b^{2})^{1/2}$.
Clearly, $|b|\in CC(b)$ and $|-b|=|b|$. Also, if $aCb$, then $|a|C|b|$ and
$|ab|=|a||b|$.

$A$ is closed under squaring, hence it forms a  special \emph{Jordan algebra} \cite{McC} under the Jordan product $a\circ b:= \frac{1}{2}(ab+ba)=\frac{1}{2}((a+b)^2-a^2-b^2)\in A$ for $a,b\in A$. The \emph{quadratic mapping} on $A$ is defined by $b\mapsto aba=2a\circ(a\circ b)-(a\circ a)\circ b$. For every $a\in A$, this mapping is linear and order preserving \cite[Theorem 4.2]{Fsyn}.

There is a mapping $^o:A\to P$ such that $ab=0 \Leftrightarrow a^ob=0$. The element $a^o$ is called the \emph{carrier} of $a$. It turns out that $a=aa^o=a^oa$, $a^o\in CC(a)$, and $a^o$ is the smallest projection  $p\in P$ such that  $a=ap$ (equivalently, $a=pa$). If $n\in \mathbb N$, then $(a^n)^o=a^o$. This mapping is closely related with the Rickart mapping,
\cite{FPbs, Pcor}.

Every  $a\in A$ determines and is determined by its \emph{spectral resolution} $(p_{a,\lambda})_{\lambda\in {\mathbb R}}$
in $P\cap CC(a)$, where $p_{a,\lambda}:=1-((a-\lambda)^+))^o=(((a-\lambda)^{+})^o)^{\perp}$ for $\lambda\in {\mathbb R}$.
Also, $L_a:=\sup\{\lambda\in {\mathbb R}:\lambda\leq a\}=\sup\{\lambda\in {\mathbb R}:p_{a,\lambda}=0\}$, $U_a:=\inf
\{\lambda\in {\mathbb R}:a\leq\lambda\}=\inf\{\lambda\in {\mathbb R}:p_{a,\lambda}=1\}$, and
\[
a=\int_{L_a-0}^{U_a}\lambda dp\sb{a,\lambda},
\]
where the Riemann-Stieltjes type sums converge to the integral in norm. Two elements in $A$ commute iff their respective
spectral resolutions commute pairwise \cite[\S 8]{Fsyn}.

The spectral resolution of $a\in A$ is a \emph{bounded resolution of identity} in $A$, that is defined as a system
$(p_{\lambda})_{\lambda\in {\mathbb R}}$ of projections in $A$ satisfying the following conditions for $\lambda,
\lambda'\in {\mathbb R}$ (see \cite[Definition 4.1]{FPspectord} and \cite[Theorem 8.4.]{Fsyn}):

(1) There exists $0\leq K\in {\mathbb R}$ such that $p_{\lambda}=0$ if $\lambda <-K$ and $p_{\lambda} =1$ if $K\leq \lambda$.

(2) $p_{\lambda} \leq p_{\lambda'}$  if $\lambda \leq \lambda'$.

(3) $p_{\lambda}=\bigwedge_{\lambda >\lambda}p_{\lambda'}$.

Notice that condition (2) implies that the projections in a bounded resolution of identity pairwise commute. In general
it is not clear whether a bounded resolution of identity is the spectral resolution of some element in $A$, but by
\cite[Theorem 4.2]{FPspectord} it is true for Banach (norm-complete) synaptic algebras.

A \emph{morphism} of synaptic algebras (or a synaptic morphism) is a linear mapping $\phi:A_1\to A_2$, where $A_1, A_2$ are synaptic algebras, with the following properties for all $a,b\in A_1$:
\begin{enumerate}
\item[(1)] $\phi(1)=1$;
\item[(2)] $\phi(a^2)=\phi(a)^2$;
\item[(3)] $aCb\, \implies\, \phi(a)C\phi(b)$;
\item[(4)]$ \phi(a^o)\grave{}=\phi(a)^o$.
\end{enumerate}

A subset $B$ of a synaptic algebra $A$ is \emph{commutative} iff $ab=ba$ for all $a,b\in B$.  A synaptic algebra $A$ is commutative iff it is a vector lattice iff $E$ is an MV-algebra \cite{FJPvl}, iff $P$ is a Boolean algebra. If $B$ is a commutative subset of $A$, then $CC(B)$ is a commutative sub-synaptic algebra of $A$.

A \emph{state} on the synaptic algebra $A$ is defined just as it
is for any order-unit normed space, namely as a linear functional
$\rho : A\to {\mathbb R}$ that is positive ($a\in A^{+}\Rightarrow
\rho(a)\in {\mathbb R}^{+}$) and normalized ($\rho(1)=1$), \cite{FJPstat}.  The state space of $A$ and the set of extremal
states on $A$ are denoted by $S(A)$ and $Ext(S(A))$. Likewise,
$S(E)$, $Ext(S(E))$, $S(P)$, and $Ext(S(P))$ denote the states
and extremal states on the convex effect algebra $E\subseteq A$
and on the OML $P\subseteq E$.

See \cite[Corollary II.1.5 and Proposition II.1.7]{Alf} for a proof of
the following theorem.

\begin{theorem} \label{th:FnlProps}
Let $(V,v)$ be an order-unit normed space and let $\rho : V\to {\mathbb R}$
be a nonzero linear functional on $V$. Let $S(V)$ denote the set of states on $V$.
Then{\rm:}
\begin{enumerate}
\item  If $a\in V$, then $a\in V^{+}$ iff $0\leq\rho(a)$ for all
 $\rho\in S(V)$.
\item If $a\in V$, then $\|a\|=\sup\{|\rho(a)|:\rho\in S(V)\}$.
\item $\rho$ is positive iff it is bounded with $\|\rho\|=\rho(v)$.
\item $\rho\in S(V)$ iff $\rho$ is bounded and $\|\rho\|=\rho(v)=1$.
\end{enumerate}
\end{theorem}

Recall that there is an affine bijection $\rho \leftrightarrow s$ between states  $\rho \in S(A)$ and states $s\in S(E)$ via extension and restriction, moreover $\rho \in Ext S(A)$ iff $s\in ExtS(E)$.

A state $\rho$  on $A$ is a \emph{normal state} iff $0\leq a_{\alpha} \nearrow a \ \implies \ \rho(a_{\alpha}) \nearrow \rho(a)$. A set $S\subseteq S(A)$ of states is \emph {separating} iff $\rho(a)=0$ for all $\rho\in S$ implies $a=0$, and $S$ is \emph{ordering} if $\rho(a)\leq \rho(b) \, \forall \rho \in S$ implies $a\leq b$.

It was proved in \cite{FPbs, Pcor} that a norm-closed synaptic algebra,  hence  a \emph{Banach synaptic algebra}, is isomorphic to a special  Rickart JB-algebra, hence a Rickart JC-algebra. Recall that a JC-algebra is by definition a  norm-closed subalgebra of the algebra of all bounded self-adjoint operators on a complex Hilbert space, which is closed  under the Jordan product $a\circ b:=\frac{1}{2}(ab+ba)$ \cite{ASS}.

Recall that a JC-algebra $C$ has the \emph{Rickart property} if, for every $b\in C$, there exists a projection $p\in C$ such that, for all $g\in C$, $bg=0 \, \Leftrightarrow g=pg$. Clearly, $b^o=1-p$ is the carrier of $b$.
A JC-algebra $C$ is a \emph{Rickart JC-algebra} if it has the Rickart property \cite{Pcor}.

By \cite{Sh}, A JB-algebra is a dual Banach space iff it is monotone complete and admits a separating set of normal states; a JB-algebra satisfying these equivalent conditions is called a \emph{JBW-algebra}. Moreover, every JBW-algebra admits a direct decomposition into a special and exceptional part. Consequently, a synaptic algebra which  satisfies the above conditions is a  weakly closed Jordan operator algebra (a so-called \emph{JW-algebra}), see \cite{T}.

\section{Effect algebras}

An \emph{effect algebra} \cite{FoBe} is a set $L$ with two distinguished elements $0,1$ and with a partial binary operation $\oplus :L \to L$ such that for all $a,b,c\in L$ the following holds:

\begin{enumerate}
\item[{\rm(EA1)}] if $a\oplus b$ is defined then $b\oplus a$ is defined and $a\oplus b=b\oplus a$ (commutativity);
\item[{\rm(EA2)}] If $b\oplus c$  and $a\oplus(b\oplus c)$ are defined then $a\oplus b$ and $(a\oplus b)\oplus c$ are defined and $a\oplus(b\oplus c)=(a\oplus b)\oplus c$ (associativity);
\item[{\rm EA3)}] for every $a\in L$ there is a unique $a'\in L$ such that $a\oplus a'=1$ (orthosupplementation);
\item[{\rm(EA4)}] if $1\oplus a$ is defined then $a=0$ (zero-one law).
\end{enumerate}

We will write $L=(L,\oplus, 0,1)$  for effect algebra.
Elements $a,b\in L$ are \emph{orthogonal} (written $a\perp b$) iff $a\oplus b$ is defined in $L$. In what follows, we often write $a\oplus b$ tacitly assuming that $a\perp b$.  A partial ordering is defined on $L$ as follows: $a\leq b$ iff there is $c\in L$ such that $a\oplus c=b$. The element $c$ is uniquely defined, and we write $c=b\ominus a$.
It is easy to check that $a\perp b$ iff $a\leq b'$.

The operation $\oplus$ can be extended to finite number of elements by recurrence in an obvious way. Owing to (EA2)
we may omit parentheses in the expressions of the form $a_1\oplus a_2\oplus \cdots \oplus a_n$.

A family $\{a_i: i\in I\}$, where $I$ is an arbitrary set, is called \emph{orthogonal} iff every finite subfamily of it admits an $\oplus$-sum (or \emph{orthosum}) in $L$. If the element $a=\bigvee_{F\subseteq I} \oplus_{i\in F} a_i$ exists in $L$, where the supremum is taken over all finite subsets  $F$ of $I$, then $a$ is called the \emph{orthosum} of the orthogonal family $\{ a_i:i\in I\}$, and is denoted by $a:=\bigoplus_{i\in I} a_i$.

An effect algebra $L$ is called \emph{orthocomplete} iff the orthosum exists for any orthogonal family of its elements, and $L$ is called \emph{$\sigma$-orthocomplete} iff the orthosum exists for every countable orthogonal family of its elements.

A mapping $s:L\to [0,1]$ from $L$ to the interval $[0,1]$ of real numbers is a \emph{state} on $L$ if  (i) $s(1)=1$; (ii) $s(a\oplus b)=s(a)+s(b)$ whenever $a\oplus b$ exists in $L$. A state $s$ is said to be  \emph{$\sigma$-additive } or \emph{completely additive} iff $s(\bigoplus_{i\in I}a_i)=\sum_{i\in I} s(a_i)$ holds for any countable or arbitrary index set $I$ such that $\bigoplus_{i\in I} a_i$ exists in $L$.

A nonempty set $S$ of states on $L$ is \emph{ordering} iff, for $a,b\in L$, $a\leq b$ iff $s(a)\leq s(b)$ for all $s\in S$:
$S$ is \emph{separating} iff $s(a)=s(b)$ for all $s\in S$ implies $a=b$.
Notice that we may always replace  $S$ by its ($\sigma$)-convex hull $Conv(S)$.

A mapping $\phi:L_1\to L_2$, where $L_1,L_2$ are effect algebras, is a \emph{morphism} if (i) $\phi(1)=1$, (ii)$ \phi(a\oplus b)=\phi(a)\oplus \phi(b)$ whenever $a\oplus b$ exists in $L_1$. A morphism $\phi$ is a \emph{$\sigma$-morphism}  (\emph{complete morphism}) iff it preserve all existing countable (arbitrary)  $\oplus$-sums. A bijective morphism such that $a\perp b$ iff $\phi(a)\perp \phi(b)$ is an \emph{isomorphism}. A $\sigma$-isomorphism, resp. complete isomorphism is defined in an obvious way.

A subset $M$ of an effect algebra $L$ is a \emph{sub-effect algebra} iff  (i) $0\in M$; (ii) $a,b\in M$, $a\perp b$ implies $a\oplus b\in M$; (iii) $a\in M$ implies $a'\in M$.

Important examples of effect algebras, so-called \emph{interval effect algebras} are obtained as follows. Let $(G,G^+)$ be a partially ordered abelian group, and let $u\in G^+$. the interval $G^+[0,u]:=\{ g\in G: 0\leq u\}$ endowed with a partial operation $\oplus$ such that, for $g,h\in G^+[0,u]$,  $g\oplus h$ is defined iff $g+h\leq u$ and then $g\oplus h=g+h$, is an effect algebra. For $G={\mathcal B}(H)^{sa}$, the self-adjoint part of the group of all bounded selfadjoint operators on a Hilbert space $H$, we obtain the \emph{effect algebra of Hilbert space effects}.

An effect algebra $L$ is  \emph{convex} \cite{GP} if it bears a \emph{convex structure}, i.e., there is a mapping $(\lambda,a) \mapsto \lambda a$ from ${\mathbb R}[0,1]\times L \to L$ such that
\begin{enumerate}
\item[(C1)] If $\alpha, \beta \in {\mathbb R}[0,1]$ and $a\in L$, then $\alpha(\beta a)=(\alpha \beta)a$.
\item[(C2)] If $\alpha, \beta \in {\mathbb R}[0,1]$ with $\alpha +\beta \leq 1$ and $a\in L$, then $\alpha a\perp \beta a$ and $(\alpha +\beta)a=\alpha a\oplus \beta a$.
\item[(C3)] If $a,b\in L$ with $a\perp b$ and $\lambda \in {\mathbb R}[0,1]$, then
$\lambda a \perp \lambda b$ and $\lambda(a\oplus b)=\lambda a\oplus \lambda b$.
\item[(C4)] If $a\in L$, then $1a=a$.
\end{enumerate}

If $(V,V^+)$ is an ordered vector space with a positive cone $V^+$ and $u\in V^+$, then the interval $V[0,u]$ is a convex effect algebra, called \emph{linear effect algebra}. The following theorem implies that every convex effect algebra is linear.

Recall that an ordered vector space $(V,V^+)$ with $u\in V^+$ is \emph{generated by the  interval} $V[0,u]=\{ v\in V: 0\leq v\leq u\}$ if every element $v\in V^+$ is a finite linear combination of elements of $V[0,u]$, and $V$ is directed, i.e., $V=V^+-V^+$.

\begin{theorem}\label{th:linear}
{\rm(1)} If $L$ is a convex effect algebra, then $L$ is affinely isomorphic to a linear effect algebra $V[0,u]$  that generates an ordered linear space $(V,V^+)$ {\rm \cite[Theorem 3.1]{GP}}.

{\rm(2)} If $L$ is a convex effect algebra with corresponding linear effect algebra $V[0,u]$ that generates $(V,V^+)$, then $(V,V^+,u)$ is an order unit space if and only if $L$ possesses an ordering set of states \rm{\cite[Theorem 3.6]{GPBB}}.
\end{theorem}

An effect algebra that forms a lattice is called a \emph{lattice-effect algebra}. A lattice effect algebra $M$ in which every two elements $a,b$ are compatible (equivalently,  $(a\vee  b)\ominus a = b\ominus (a\wedge b)$ for all $a,b\in M$) is called an \emph{MV-effect algebra}. As a lattice, an MV-effect algebra is distributive. Notice that a convex effect algebra is an MV-effect algebra iff it is lattice ordered \cite{GPBB}. MV-effect algebras are closely related to MV-algebras introduced by Chang \cite{Chang}. Every MV-effect algebra can be organized into an MV-algebra, and reciprocally, an MV-algebra can be organized into
an MV-effect algebra (see e.g. \cite{DvPu}).  By a result of Mundici \cite{Mun}, there is a categorical equivalence between MV-algebras and lattice ordered groups.

Let $L$ be a $\sigma$-orthocomplete effect algebra, and $(X, {\mathcal A})$ a measurable space. By an \emph{$(X, {\mathcal A})$-observable} on $L$ we mean a mapping $\xi:{\mathcal A}\to L$ such that

\begin{enumerate}
\item[{\rm(i)}] $\xi(X)=1$;
\item[{\rm(ii)}] the system $(\xi(A_i))_{i\in {\mathbb N}}$ is orthogonal and $\xi(\cup_{i=1}^{\infty}A_i)=\bigoplus_{i=1}^{\infty}\xi(A_i)$ whenever
    $A_i\in {\mathcal A}$ for $i\geq 1$ and $A_i\cap A_j=\emptyset$, $i\neq j$.
\end{enumerate}

If $(X, {\mathcal B})\equiv ({\mathbb R}, {\mathcal B}({\mathbb R}))$, then $\xi:  {\mathcal B}({\mathbb R}) \to L$ is a \emph{real} observable.

Let $(X_1, {\mathcal A}_1)$ be another measurable space and let $f:X \to X_1$ be a function such that $f^{-1}(A)\in {\mathcal A}$ whenever $A\in {\mathcal A}_1$. If $\xi: {\mathcal A}\to L$ is an observable, then $f\circ \xi : A\mapsto \xi(f^{-1}(A)), A\in {\mathcal A}_1$ is an $(X_1,{\mathcal A}_1)$-observable on $L$, which is called the \emph{$f$-function} of $\xi$ denoted by $f(\xi)$.

If $\xi$ is an $(X,{\mathcal A})$-observable on $L$, and $s$ is a $\sigma$-additive state on $L$, then $s_{\xi}:=s\circ \xi:{\mathcal A}\to [0,1]$ is a probability measure on $(X, {\mathcal A})$. If $\xi$ is a real observable, we denote by
\[
s(\xi):=\int_{\mathbb R} ts_{\xi}(dt)
\]
the \emph{expectation} of $\xi$ in $s$ whenever the right-hand side of the above equation exists and is finite. Using the integral transformation theorem  we obtain for any Borel function  $f:X \to {\mathbb R}$ and for any $(X,{\mathcal A})$-observable on $L$,
\begin{eqnarray*}
s(f(\xi))&=&\int_{\mathbb R}us(f(\xi(du))\\
&=& \int_{\mathbb R} us(\xi(f^{-1}(du)))=\int_Xf(t)s_{\xi}(dt)).
\end{eqnarray*}

For an $(X,{\mathcal A})$-observable $\xi$ on $L$, let ${\mathcal R}(\xi):=\{ \xi(A):A\in {\mathcal A}\}$ denote the range of $\xi$. An observable $\xi$ is called \emph{sharp} if its range consists of sharp elements (recall that an element $a\in L$ is sharp if $a\wedge a'=0$). For example, in the effect algebra ${\mathcal E}(H)$ of Hilbert space effects, sharp elements are projections,  observables  are  POV-measures, and  sharp observables are PV-measures.

\subsection{Weak Markov kernels and smearings of observables}

Let $(X,{\mathcal A})$ and $(Y,{\mathcal B})$ be measurable spaces. A mapping $\nu:X\times {\mathcal B} \to [0,1]$ is a \emph{Markov kernel} if the following is satisfied:
\begin{enumerate}
\item[(i)] for any fixed $x\in X$, $\nu_x(.):=\nu(x,.):{\mathcal B} \to [0,1]$ is a probability measure;
\item[(ii)] for any fixed $B\in {\mathcal B}$, the mapping $x\mapsto \nu_B(x):=\nu(x,B)$ is ${\mathcal A}$-measurable.
\end{enumerate}

The notion of Markov kernel has  been weakened in \cite{JPV1} to so-called \emph{weak Markov kernel}.  For our purposes, we introduce an even more general form of a weak Markov kernel as follows. Let $(X,{\mathcal A})$ and $(Y,{\mathcal B})$ be measurable spaces, and let $I$ be a $\sigma$-ideal of the $\sigma$-algebra $\mathcal A$.
We will say that property $Q$ holds $I$-almost everywhere ($I$-a.e., for short), if $\{x\in X: Q \ \mbox{does not hold in x} \} \in I$.
Let $\nu:X\times {\mathcal B} \to {\mathbb R}$. We say that  $\nu$ is a \emph{weak Markov kernel with respect to} $I$  if
\begin{enumerate}
\item[(i)] $x\mapsto \nu(x,B)$ is ${\mathcal A}$-measurable  for all $B\in {\mathcal B}$;
\item[(ii)] for every $B\in {\mathcal B}$, $0\leq \nu(x,B)\leq 1$  $I$-a.e.;
\item[(iii)] $\nu(x,Y)=1$ ${\mathcal P}$-a.e., and $\nu(x,\emptyset)=0$ $I$-a.e.;
\item[(iv)] if $(B_n)_{n\in {\mathbb N}}$ is a sequence in ${\mathcal B}$ such that $B_n\cap B_m=\emptyset$, $n\neq m$, then
\[
\nu(x,\bigcup B_n)=\sum_n\nu(x,B_n), \ {I}- a.e.
\]
\end{enumerate}
If $\nu,\mu: X\times \mathcal B\to [0,1]$ are weak Markov kernels with respect to $I$, then we say that $\nu\sim_I\mu$
if for all $B\in \mathcal B$, $\{x, \nu(x,B)\ne \mu(x,B)\}\in I$. Clearly, $\sim_I$ is an equivalence relation.

Note that in the case when $M_1^+(X,{\mathcal A})$ denotes the set of all probability measures on $(X,{\mathcal A})$,  ${\mathcal P}\subseteq M_1^+(X,{\mathcal A})$, and we put $I_{\mathcal P}:=\{ x\in X: \mu(x)=0 \forall \mu\in {\mathcal P}\}$, then the definition of $I_{\mathcal P}$-weak Markov kernel coincides with the definition of a weak Markov kernel with respect to $\mathcal P$ in \cite{JPV1}.  Clearly, a weak Markov kernel with respect to the whole $M_1^+(X,{\mathcal A})$ is a Markov kernel.

It is easy to see that if $\nu$ is a weak Markov kernel with respect to $\mathcal P$, then
\begin{equation}\label{eq:1}
\nu(P)(B):=\int_X\nu(x,B)P(dx), \ B\in {\mathcal B}
\end{equation}
is a probability measure on $\mathcal B$ for all probability measures $P\in {\mathcal P}$.

Let $L$ be a $\sigma$-orthocomplete effect algebra with a ($\sigma$-convex) ordering set ${\mathcal S}$ of $\sigma$-additive states and let $(X,{\mathcal A})$ be a measurable space.
Every observable $\xi: {\mathcal A} \to L$ can be characterized by
its probability distribution  $\Phi_{\xi}: {\mathcal S} \to M_1^+(X,{\mathcal A})$ defined by
\begin{equation}\label{eq:2}
\Phi_{\xi}(m) (A)=m\circ \xi(A), \, m\in {\mathcal S}, A\in {\mathcal A}.
\end{equation}

\begin{definition}\label{de:smearing} Let $L$ be a $\sigma$-orthocomplete effect algebra  with a nonempty set of $\sigma$-additive states ${\mathcal S}_{\sigma}(L)$.
 Let $(X,{\mathcal A})$ and $(Y,{\mathcal B})$ be measurable spaces. Let $\xi$ be a $(X,{\mathcal A})$-observable on $L$.  Put $I_{\xi}=\{ A\in {\mathcal A}: \xi(A)=0\}$.
 If $\nu: X\times {\mathcal B} \to {\mathbb R}$ is a weak Markov kernel with respect to $I_{\xi}$,
 and there is an observable $\eta$ such that for every $B\in {\mathcal B}$ and every $m\in {\mathcal S}_{\sigma}(L)$,
 \[
 \Phi_{\eta}(m)(B)= m(\eta(B))=\int \nu(x,B)m\circ \xi(dx)
\]
then $\eta$ is called a \emph{smearing} (or a \emph{fuzzy version}) of $\xi$ (with respect to $\nu$).
\end{definition}

We  note that if $S_{\sigma}(L)$ is a (nonempty) set of $\sigma$-additive states on $L$, then $B \mapsto \nu(m\circ \xi)(B)$, where
\[
\nu(m\circ \xi)(B):=\int_X \nu(x,B)m\circ \xi(dx), B\in {\mathcal B},
\]
is a probability measure on $(Y,{\mathcal B})$, and if $\eta$ is a smearing of $\xi$, then
\[
 \nu(m\circ \xi)(B)=m(\eta(B)), B\in {\mathcal B}.
\]
It is clear that if $\mu\sim_{I_\xi}\nu$, then the smearings with respect to $\mu$ are the same as those with respect to
$\nu$.

Observe that a smearing of an observable need not be unique nor  exist at all. Note also that the function $f\circ \xi$
defined above
is a  special type of smearing, with respect to the Markov kernel $\nu(x,B)=\chi_{f^{-1}(B)}(x)$.

%
%
%
%

\begin{theorem}\label{th:smearing} Let $L$ be a convex $\sigma$-orthocomplete effect algebra with an ordering set
$\mathcal S_\sigma(L)$ of
$\sigma$-additive states and let  $\xi$ be an $(X,\mathcal A)$  observable on $L$. Let $(Y,\mathcal B)$ be a measurable space and let $\nu: X\times \mathcal B\to [0,1]$ be a weak
Markov kernel with respect to $I_\xi$. Then there is a unique smearing of $\xi$ with respect to $\nu$.
\end{theorem}

\begin{proof} We first define integrals with respect to $\xi$, in the following sense. Let $f:X\to [0,1]$ be $\mathcal
A$-measurable. We will show that there is an element $\xi(f)\in L$ such that for all $m\in \mathcal S_\sigma(L)$,
 we have
\[
m(\xi(f))=\int_X f(x) m\circ \xi(dx).
\]

Since $\mathcal S_\sigma(L)$ is ordering, it is clear that such an element must be unique. First, let $f=\chi_\Delta$ for
 $\Delta\in A$, it this case, we put $\xi(f):=\xi(\Delta)$. Next, let $f=\sum_i c_i \Delta_i$ be a simple function, then
by standard arguments, we may suppose that $\Delta_i$ are pairwise disjoint and $c_i\in [0,1]$. Put
$\xi(f):=\sum_ic_i\xi(\Delta_i)$. Since $L$ is convex and $\oplus_i \xi(\Delta_i)$ exists in $L$, we see that $\xi(f)\in
L$, moreover, for $m\in \mathcal S_\sigma(L)$,
\[
m(\xi(f))=\sum_i c_im(\xi(\Delta_i))=\int_X f(x)m\circ\xi(dx).
\]
If $f:X\to [0,1]$ is a measurable function, then there is an increasing  sequence  of simple functions $f_n:X\to [0,1]$
converging pointwise to $f$. Since $\mathcal S_\sigma(L)$ is ordering and we have
\[
m(\xi(f_n))=\int f_n dm\circ\xi\le \int f_{n+1} dm\circ\xi=m(\xi(f_{n+1})),\qquad \forall m\in \mathcal S_\sigma(L),
\]
it follows that $\xi(f_n)\le \xi(f_{n+1})$. Since $L$ is $\sigma$-orthocomplete, there is some element
$\xi(f)\in L$ such that $\vee_n \xi(f_n)=\xi(f)$.
Using  Lebesgue monotone convergence theorem, we have for $m\in \mathcal S_\sigma(L)$,
\[
m(\xi(f))=\bigvee_n m(\xi(f_n))=\lim_n m(\xi(f_n))= \int_X f(x) m\circ\xi(dx).
\]
By uniqueness, it is clear that $\xi(f)$ does not depend on the choice of the sequence $f_n$. Note also that
 if $f':X\to [0,1]$ is such that $\xi(\{x\in X, f(x)\ne f'(x)\})=0$, then $\xi(f)=\xi(f')$.

We now define $\eta(B)=\xi(\nu_B)$, $B\in \mathcal B$, where $\nu_B=\nu(\cdot,B)$.
We now show that $\eta$ is an observable, it is then clear by definition that $\eta$ must be the unique smearing of $\xi$ with respect
to $\nu$.  So let $\{B_i\}$ is such that $B_i\cap B_j=\emptyset$ for $i\ne j$ and let
$B=\cup_i B_i$. Then by the definition of weak Markov kernel, $\nu_B=\sum_i \nu_{B_i}$ up to some set $\Delta_0$ with
$\xi(\Delta_0)=0$ and hence
\[
\eta(B)=\xi(\nu_B)=\xi(\sum_i \nu_{B_i})=\sum_i \eta(B_i),
\]
the last equality holds because
\[
m(\xi(\sum_i \nu_{B_i}))=\int_X\sum_i \nu_{B_i}(x)m\circ\xi(dx)=\sum_i \int_X\nu_{B_i}m\circ \xi(dx)=\sum_i
m(\xi(\nu_{B_i})).
\]
The facts that $\eta(\emptyset)=0$ and $\eta(Y)=1$ are proved similarly.

\end{proof}

\begin{remark} Note that the element $\xi(f)$ defined in the above proof is such that for each $\sigma$-additive state
$m$, $m(\xi(f))$ is the expectation of the observable $f\circ \xi$ in $m$.

\end{remark}

\section{Observables on GH-algebras}

For a measurable space $(X,{\mathcal A})$, an $(X,{\mathcal A})$-observable on a synaptic algebra $A$ is defined as an observable on the   effect algebra $E=A[0,1]$ or, in the case of sharp observables, on the  OML of projections $P$ of $A$. Observables are usually studied on $\sigma$-orthocomplete effect algebras resp. $\sigma$-complete OMLs. By the study of observables, we will therefore assume that $E$, or at least $P$, is $\sigma$-orthocomplete.

It can be easily seen that the effect algebra $E$ of a synaptic algebra $A$  is $\sigma$-orthocomplete iff the synaptic algebra $A$ is monotone $\sigma$-complete.
Indeed, assume that $E$ is $\sigma$-orthocomplete, and let $(a_n)_n$ be an ascending sequence of elements in $A$ bounded above by an element $b\in A$. Then
\[
0\leq \frac{b-a_n}{\|b-a_1\|}\leq 1,
\]
and
\[
(\frac{b-a_n}{\|b-a_1\|})_n
\]

is descending, so it has an infimum in $E$, hence  $(a_n)_n$ has a supremum in  $A$.

 We say that a state $\rho$ on $A$ is \emph{$\sigma$-normal} if for every monotone increasing sequence  $(a_n)$ of positive elements, $a_n\nearrow a$ $\implies$ $\rho(a_n)\nearrow \rho(a)$. Clearly, a state $\rho$ on $A$ is $\sigma$-normal iff its restriction to $E$ is a $\sigma$-additive state on the effect algebra $E$.

A special kind of a synaptic algebra is a \emph{generalized Hermitian  (GH-) algebra}, which was introduced and studied in \cite{FPgh, FPrgh}.

The following characterization was found in  \cite[Theorem 9.1]{FPbs}.

\begin{theorem}\label{th:ghalgebra}
A GH-algebra is the same thing as a synaptic algebra $A$ such that every bounded monotone increasing sequence $a_1\leq a_2\leq \cdots$ of pairwise commuting elements in $A$ has a supremum in $A$.
\end{theorem}

By \cite[Lemma 5.4]{FPgh}, if $A$ is a GH-algebra, then $P$ is a $\sigma$-complete OML. Clearly, if $A$ is monotone
$\sigma$-complete synaptic algebra, then it is a GH-algebra, and a commutative synaptic algebra is a GH-algebra if and
only if it is monotone $\sigma$-complete. Moreover, every monotone $\sigma$-complete synaptic algebra is norm complete
\cite[Proposition 3.9]{Hand}. For any synaptic algebra $A$, if $T\subseteq A$ and $T$ has a supremum $b$ in $A$, then
$b\in CC(T)$ \cite[Theorem 6.2]{FPbs}. Applying this to a commutative subset $T$ of a GH-algebra $A$, we obtain that
$CC(T)$, in particular $CC(a)$ for every $a\in A$,  is  monotone $\sigma$-complete, hence a commutative GH-subalgebra of
$A$. Therefore, we may consider  observables with a commutative range on a GH-algebra $A$ and we may assume that $A$ is
commutative.

We have the following representation theorem for commutative GH-algebras \cite[Theorem 6.6]{FJPstat}. Recall that for a
compact set $X$, $C(X,{\mathbb R})$ denotes the set of all continuous functions $f:X \to {\mathbb R}$.
Recall that a \emph{morphism of GH-algebras}  $\phi: A_1 \to A_2$ is defined as a synaptic morphism  with the additional property that  given a sequence of pairwise commuting elements $(a_n)_n$ such that $a_n\nearrow a$ in $A_1$, then $\phi(a_n)\nearrow \phi(a)$ in $A_2$.
In what follows, $P(X,{\mathbb R})$ denotes the set of all characteristic functions of clopen subsets of $X$.

\begin{theorem}\label{th:ghrepr} Suppose $A$ is a commutative GH-algebra and let $X$ be the basically disconnected Stone space of the $\sigma$-complete Boolean algebra $P$. Then there is an isomorphism of GH-algebras  $\Psi: A\to C(X,{\mathbb R})$ of $A$ onto $C(X,{\mathbb R})$, such that the restriction $\psi$ of $\Psi$ to $P$ is a boolean isomorphism of $P$ onto $P(X,{\mathbb R})$ as per Stone's representation theorem.
\end{theorem}

A functional calculus for continuous functions on GH-algebras is defined as follows. Let $f\in C(spec(a),{\mathbb R})$ and let $g:=\Psi(a)\in C(X,{\mathbb R})$. Then $spec(a)=\{ g(x): x\in X\}$, $f\circ g\in C(X,{\mathbb R})$, and we define the element $f(a)\in CC(a)$ by $f(a):=\Psi^{-1}(f\circ g)$.
 In particular, if $q(t)=\alpha_0+\alpha_1t+\alpha_2t^2+\cdots +\alpha_nt^n$, then $q(a)=\alpha_0+\alpha_1g(x)+\alpha_2g(x)^2+\cdots +\alpha_ng(x)^n$.
We have $f(a)=\int_{L_a-0}^{U_a}f(\lambda) dp_{a,\lambda}$, $f\in C(spec(a),{\mathbb R})$, \cite[Theorem 7.7]{FJPls}.

Notice that countable suprema in $C(X,{\mathbb R})$ are not the pointwise suprema of functions. This can be improved by the following version of the Loomis-Sikorski theorem \cite[Theorem 6.6]{FJPls}, which is an extension of the Loomis-Sikorski theorem for $\sigma$-MV algebras (cf. \cite{DvPu}).  For the definition of a gh-tribe see \cite{FJPls}. In short, a gh-tribe  $\mathcal T$ is a commutative GH-algebra consisting of bounded real-valued functions on a nonempty set $X$ with pointwise ordering and supremum norm. The set of characteristic functions in $\mathcal T$ forms a $\sigma$-field  ${\mathcal B}({\mathcal T})$ of subsets of $X$.
By \cite{BuKl}, every function $f\in {\mathcal T}$ is ${\mathcal B}({\mathcal T})$-measurable. Moreover, for every $\sigma$-normal state on ${\mathcal T}$, we have
$s(f)=\int_Xf(t)P(dt)$, where $P:=s/{\mathcal B}({\mathcal T})$ is a probability measure on ${\mathcal B}({\mathcal T})$.

\begin{theorem}\label{th:ls} {\bf Loomis-Sikorski theorem.}  Let $A$ be a commutative GH-algebra and let $X$ be the basically disconnected Stone space for the $\sigma$-complete Boolean algebra $P$ of projections in $A$. Then there exists a gh-tribe ${\mathcal T}$ on $X$ such that $C(X,{\mathbb R})\subseteq {\mathcal T}$ and there exists a surjective morphism $h$ of GH-algebras from ${\mathcal T}$ onto $A$.
\end{theorem}

Using Theorem \ref{th:ls}, it was shown that each element $a$ in a GH-algebra $A$ corresponds to a sharp real observable $\xi_a$ on the $\sigma$-OML $P$ of projections on $A$ \cite[Theorem 7.4]{FJPls}. Every $a\in A$ is contained in a commutative subalgebra $B$ (we may put, e.g., $B=CC(a)$) of $A$.  This $B$ is a commutative GH-algebra in its own right, and admits a Loomis-Sikorski representation $(X,{\mathcal T},h)$ by Theorem \ref{th:ls}.  By definition, $\xi_a(B)=h(f_a^{-1}(B))$, $B\in {\mathcal B}({\mathbb R})$, where $f_a$ is a function in ${\mathcal T}$ such that $h(f_a)=a$.
The observable $\xi_a$ is independent on the choice of the function $f_a$, and is the unique real observable on $P$ such
that $\xi_a((-\infty, \lambda ])=p_{a,\lambda}$ for all $\lambda \in {\mathbb R}$, where $\{p_{a,\lambda}\}$ is the
spectral resolution of $a$. Since every element in $A$ determines and is uniquely determined  by its spectral
resolution, the observable $\xi_a$ is uniquely determined by $a$, and also determines $a$. The observable $\xi_a$ is
bounded, in the sense that there is some $K\ge 0$ such that $\xi((-K,K))=1$.

Since elements $a$ and $b$ in $A$ commute iff their respective spectral resolutions pairwise commute, we obtain that $aCb$ iff the ranges of $\xi_a$ and $\xi_b$ pairwise commute, i.e., iff $\xi_a$ and $\xi_b$ are compatible observables on the OML $P$ (cf. \cite{Var}).
Moreover, for every $\sigma$-normal state $\rho$ on $A$,
\[
\rho(a)=\int_{\mathbb R}\lambda \rho(\xi_a(d\lambda)), \ a\in A.
\]

The following result shows a one-to-one correspondence between bounded sharp real observables and elements of a GH-algebra.

\begin{theorem}\label{thm:sharp} For any bounded sharp real observable $\xi$ on a GH-algebra $A$, there is a unique element $a\in A$
 such that $\xi=\xi_a$.

\end{theorem}

\begin{proof} Define $p_\lambda= \xi((-\infty,\lambda])$, $\lambda\in \mathbb R$. Then $\{p_\lambda\}$ is a bounded
resolution of identity. Indeed, we  have
\[
p_\lambda= \bigwedge_{\lambda'>\lambda, \lambda'\in \mathbb Q}p_{\lambda'}\ge \bigwedge_{\lambda' >\lambda} p_{\lambda'}
\ge p_\lambda
\]
where the equality follows by $\sigma$-additivity of $\xi$. The other two properties of a bounded resolution of identity
are clear. Since all $p_\lambda$ commute, they are contained in a commutative GH-subalgebra $B$ in $A$.
Since $B$ is monotone $\sigma$-complete, it is a Banach GH-algebra. By \cite[Theorem 4.2]{FPspectord}, $\{p_\lambda\}$
is the spectral resolution of an element $a\in B$. It follows that
$\xi_a((-\infty,\lambda])=p_\lambda=\xi((-\infty,\lambda])$, hence $\xi=\xi_a$.

\end{proof}

The functional calculus on commutative GH-algebras can be extended to all Borel measurable functions. Indeed, by \cite[Proposition 7.1.11]{DvPu}, and \cite[Proposition 7.1.25]{DvPu}, the gh-tribe generated by $C(X,{\mathbb R})$ on a basically disconnected  compact Hausdorff space $X$ coincides with the set of bounded Baire measurable functions on $X$. Notice that the Baire $\sigma$-algebra is the $\sigma$-algebra generated by compact $G_{\delta}$ sets on $X$, or equivalently, by $\{ f^{-1}([\alpha, \infty)):f\in C(X,{\mathbb R}), \alpha \in {\mathbb R}\}$.
Now let $A$ be a commutative GH-algebra, and let $a\in A$ be such that $g=\Psi(a)\in C(X,{\mathbb R})$, where $\Psi$ is the isomorphism from Theorem \ref{th:ghrepr}. Let $f$ be a real-valued Borel function defined on $\sigma(a)=\{ g(x):x\in X\}$. Let $B\in {\mathcal B}({\mathbb R})$ be any Borel set. Then  $(f\circ g)^{-1}(B)= g^{-1}(f^{-1}(B))$, where $f^{-1}(B)$ is a Borel set, hence $(f\circ g)^{-1}(B)$ belongs to the Baire $\sigma$-algebra. It follows that $f\circ g$  is  Baire measurable, equivalently, ${\mathcal B}({\mathcal T})$-measurable function. It follows that $f\circ g\in {\mathcal T}$, and we may define $f(a):= h(f\circ g)$, where $h$ is the homomorphism from Theorem \ref{th:ls}. Notice also, that instead of $g$ we may use any function $f_a\in {\mathcal T}$ such that $h(f_a)=a$. Indeed, we have $h(f_a)=h(g)$ iff the set $\{ x\in X: f_a(x)\neq g(x)\}$ is meager. But then $\{ x\in X: f\circ f_a(x)\neq f\circ g(x)\}$ is meager too, and so $h(f\circ g)=h(f\circ
 f_a)$. For any $\sigma$-normal state we have
\[
\rho(f(a))=\int_{\mathbb R}f(\lambda)\rho(\xi_a(d\lambda))),
\]
where $\xi_a$ is the observable corresponding $a$, i.e.,  $\xi_a(B))=h(g^{-1}(B))$.

Following \cite{Var}, we can also define functions of several commuting
elements.
Let $a_1,a_2,\ldots, a_n \in A$ be pairwise commuting, and let $\xi_1,\xi_2,\ldots, \xi_n$ be their corresponding sharp
real observables. Then these observables are compatible observables on the OML $P$.   Let $A_1$ be the smallest commutative sub-synaptic algebra containing all $a_1,a_2,\ldots, a_n$. Then  the ranges of $\xi_i, i=1,2,\ldots, n$ are contained in $P_1:=A_1\cap P$, which is the smallest Boolean subalgebra of the OML $P$ containing them. Let $(X, {\mathcal T}, h)$ be the Loomis-Sikorski representation of $A_1$ according Theorem \ref{th:ls}. Then $h$ maps ${\mathcal B}({\mathcal T})$ onto $P_1$. We may consider observables $\xi_i, i=1,2,\ldots, n$ as observables from ${\mathcal B}({\mathbb R})$ to $P_1$. Let $f_i, i=1,2,\ldots, n$ be functions in ${\mathcal T}$ such that $h(f_i)=a_i$, so that $\xi_i(B)=h(f_i^{-1}(B))$, $i=1,2,\ldots, n$, for all $B\in {\mathcal B}({\mathbb R})$.

 We will follow the construction in the proof of \cite[Theorem 1.6 (ii)]{Var}.
Define $F:X\to {\mathbb R}^n$ by $F(x)=(f_1(x),f_2(x),\ldots,f_n(x))$. Then $F$ is ${\mathcal B}({\mathcal
T})$-measurable. Let $u:=h\circ F^{-1}$. Then $u: {\mathcal B}({\mathbb R}^n)\to P$ is a $\sigma$-morphism such that
$\xi_i(B)=u(\pi_i^{-1}(B))$, $i=1,2,\ldots, n$, for all $B\in {\mathcal B}({\mathbb R}$. Here $\pi_i(t_1,t_2,\ldots,
t_n)\to t_i$ is a projection of ${\mathbb R}^n\to {\mathbb R}^1$. Since ${\mathcal B}({\mathbb R})$ is the smallest
$\sigma$-algebra of subsets of ${\mathbb R}^n$ containing all  $\pi_i^{-1}(B), i=1,2,\ldots, n$, the range of $u$ is
$P_1$. The uniqueness of $u$ is obvious. For any Borel function $G: {\mathbb R}^n\to {\mathbb R}$, the mapping $u\circ
G^{-1}$ is an observable on $P$ whose range is contained in $P_1$. We then define the function $G$ of the observables
$\xi_1,\xi_2,\ldots,\xi_n$ as the observable $u\circ G^{-1}$, i.e., $G(\xi_1,\xi_2,\ldots, \xi_n):=u\circ G^{-1}$.
We have the following.

\begin{theorem}\label{th:plustimes} Let $a,b$ be commuting elements in a GH-algebra $A$.  The following statements hold (cf. also \cite{Pu1, Pu2}):
\begin{enumerate}
\item[{\rm(i)}] The observable $\xi_{a+b}$ is the $G$-function of the observables $\xi_a$ and $\xi_b$, where $G(t_1,t_2)=t_1+t_2$.
\item[{\rm(ii)}] The observable $\xi_{ab}$ is the $G$-function of the observables $\xi_a$ and $\xi_b$, where $G(t_1,t_2)=t_1.t_2$.
\end{enumerate}
\end{theorem}

\begin{proof}  (i) We have $\xi_{a+b}=h\circ f_{a+b}^{-1}$, where $f_{a+b}\in {\mathcal T}$ is any function such that $h(f_{a+b})=a+b$. Since $h$ is a morphism of GH-algebras, we have
$h(f_a+f_b)= h(f_a)+h(f_b)= h(f_{a+b})$. For every $B\in {\mathcal B}({\mathbb R})$ we have,

\begin{eqnarray*}
G(\xi_a,\xi_b)(B) &=& u\circ G^{-1}(B)=h\circ F^{-1}(G^{-1}(B))\\
&=& h((G\circ F)^{-1}(B))\\
&=& h((f_a+f_b)^{-1}(B))\\
&=& \xi_{a+b}(B).
\end{eqnarray*}

The proof of (ii) is similar.
\end{proof}

Let us now turn to general observables with a commuting range. We first observe that the Loomis-Sikorski theorem
provides a special sharp observable for any commutative GH-algebra $A$. Let $(X,{\mathcal T},h)$ be the Loomis-Sikorski
representation of $A$, then it is easy to see that the restriction of $h$ to $\mathcal B(\mathcal T)$ is a sharp
 observable on $A$. By the proof of the next theorem, every observable with range in $A$ is a smearing of $h$.

\begin{theorem}\label{th:commutrange} Every observable with commuting range on a GH-algebra
is defined by a smearing of a sharp observable.
\end{theorem}

\begin{proof} This proof is analogous to that of \cite[Theorem 4.4]{JPV2}. Let $\xi$ be an $(\Omega,{\mathcal
B})$-observable on a GH-algebra  such that the range $ {\mathcal R}(\xi)=\{ \xi(B): B\in {\mathcal B}\}$ consists of
pairwise commuting elements. Then $CC({\mathcal R}(\xi))$ is a commutative GH-algebra. Let
$(X,{\mathcal T},h)$ be the Loomis-Sikorski representation.
By Theorem  \ref{th:ls}, for every $B\in {\mathcal B}$ there is an $f_B\in {\mathcal T}[0,1]$  with $h(f_B)=\xi(B)$, where $f_B$ is ${\mathcal B}({\mathcal T})$-measurable and is unique up to $h$-null sets.
Define $\nu:X\times {\mathcal B} \to [0,1]$ by $\nu(x,B)=f_B(x)$. It can be proved
that $\nu(X,B)$ is a weak Markov kernel with respect to $I_h:=\{ B\in {\mathcal B}({\mathcal T}):h(B)=0\}$. Indeed, we have
$\xi(B)= h(\nu_B)$, $B\in {\mathcal B}$. Let $(B_i)_i$ be a disjoint sequence of elements of ${\mathcal B}$, and put
$B=\bigcup_i B_i$. Then
\[
h(f_B)=\xi(B)=\oplus_i\xi(B_i)=\sum_i h(f_{B_i})=h(\sum_i f_{B_i})
\]
and hence $\nu(x,B)=f_B(x)=\sum_i f_{B_i}(x)=\sum_i \nu(x,B_i)$, $I_h$-a.e. This proves property (iv) in the definition of a weak Markov kernel, the remaining properties are obvious.

Owing to \cite{BuKl} we have for every $\sigma$-additive state $m$ on $E$,
\[
m(\xi(B))=m (h(f_B))= \int_X f_B(x)m\circ h(dx)=\int_X \nu(x,B)m\circ\xi(dx).
\]
By Definition \ref{de:smearing}, the observable $\xi$ is  a smearing of the $(X,\mathcal B(\mathcal T))$-observable $h$.

\end{proof}

Note that for $a\in A$, the corresponding sharp real observable $\xi_a$ is the smearing of $h$ by the  Markov
kernel $\nu(x,B)=\chi_{f_a^{-1}(B)}(x)$. By Theorem \ref{thm:sharp}, any bounded sharp real observable has this form for
some function $f\in \mathcal T$.

\begin{remark} Let $A$ be a commutative GH-algebra with the Loomis-Sikorski representation $(X,{\mathcal T},h)$. Let $(Y,\mathcal B)$ be a measurable space and let $\nu: X\times \mathcal B\to [0,1]$ be a weak Markov
kernel with respect to $I_h$. Then $\xi(B):=h(\nu_B)$ is an observable and if $A$ has some $\sigma$-additive states,
 then $\xi$ is a smearing of $h$ with respect to $\nu$.

\end{remark}

\section{Observables on synaptic algebras which are dual Banach spaces}

In what follows, we will consider a synaptic algebra $A$ which is the dual of a Banach space. In this case, $A$ is itself a Banach space, hence a {\em Banach synaptic algebra}. As it was proved in \cite{FPbs,Pcor}, $A$ is then isomorphic to a JC-algebra and by \cite[Corollary 2.4]{Sh}, a JC-algebra is a dual Banach space iff it is a JW-algebra (that is, a weakly closed Jordan operator algebra, see \cite{T}), equivalently, $A$ is monotone complete and has a separating set of normal states. Notice that also conversely, every JW-algebra is a synaptic algebra, and being monotone complete, it is a GH-algebra.

To describe the predual of $A$, we will need the notion of a base norm space. Let $(V,V^+)$ be an ordered vector space and let $K$ be a base of $V^+$, that is, a convex subset of $V^+$ such that every nonzero $v\in V^+$ can be uniquely written in the form $v=\lambda x$ for $\lambda>0$ and $x\in K$. Let
\[
\|v\|_K:=\inf\{\lambda+\mu, v=\lambda x-\mu y,\ \lambda,\mu\ge 0, x,y\in K\}.
\]
If $\|\cdot\|_K$ defines a norm in $V$, we say that $V$ is a base norm space, with distinguished base $K$.
 Let us remark that the dual of an order unit space is a base norm space with the base given by the set of states. Conversely, the dual of a base norm space is an order unit space such that the order unit has constant value 1 on the distinguished base, \cite{AS}.

We will make use of the following theorem \cite[Proposition 1.11]{AS}.

\begin{theorem}\label{th:kadison} If $V$ is a base-norm space with distinguished base $K$, then the restriction map $f\mapsto f/K$ is an order and norm preserving isomorphism of $V^*$ onto the space $A_b(K)$ of all real valued bounded affine functions on $K$ equipped with pointwise ordering and supremum norm,
\end{theorem}

The next result is based on \cite[Theorem 6]{Ellis} and characterizes the convex effect algebras such that the corresponding ordered vector space is an order unit space which is the dual of a Banach space. In particular, we obtain an alternative characterization of synaptic algebras which are JW-algebras.

Let $E$ be a convex effect algebra and let $S\subseteq S(E)$ be a set of states. The $\sigma(E,S)$-topology on $E$ is given by  the neighbourhoods basis consisting of the sets
\[
V(a; s_1,\dots,s_n,\epsilon):= \{ b\in E: |s_i(a)-s_i(b)|< \epsilon\}, \quad s_i\in S, i=1,2,\ldots, n, \epsilon > 0.
\]

\begin{theorem}\label{th:ellis} Let $E$ be a convex effect algebra and let $(V,V^+,u)$ be the ordered vector space with order unit $u$ such that $E$ is isomorphic to the unit interval $V[0,u]$ of $V$. Then $(V,V^+,u)$ is an order unit space which is the dual of a Banach space if and only if
$E$ is compact with respect to the $\sigma(E,S)$-topology for a separating set $S\subseteq S(E)$.
The predual is a base norm  space, whose base is an ordering set of completely additive states on $E$.
\end{theorem}

\begin{proof} Since $[0,u]$ generates $V$, every state $\rho \in S$ can be uniquely extended to a state $\hat\rho$ on $V$. Let $\tau$ be the locally convex topology defined by the seminorms $x\mapsto  |\hat\rho(x)|$, $x\in V$ for $\rho\in S$. Since $S$ is separating, this family of seminorms is separated. Indeed, let $x\in V$ be such that $\hat\rho(x)=0$ for all $\rho\in S$. Let $\lambda,\mu\ge 0$ and $a,b\in E$ be such that $x=\lambda a-\mu b$, so that
\[
\lambda \rho(a)=\mu\rho(b),\quad \forall \rho\in S.
\]
We may assume that $\mu\le \lambda$ and $\lambda>0$. Then for $\rho\in S$,
\[
\rho(a)=\frac{\mu}{\lambda}\rho(b)=\rho(\frac{\mu}{\lambda}b).
\]
It follows that $a= \frac{\mu}{\lambda}b$, hence $x=0$.

Since the restriction of $\tau$ to $E$ coincides with $\sigma(E,S)$, $E$ is compact with respect to $\tau$.
By  \cite[Theorem 6]{Ellis},  $(V,V^+,u)$ is an order unit space which is a dual of the Banach space $N$ of all $\tau$-continuous functionals on $V$, moreover, the $\sigma(V,N)$-topology agrees with $\tau$ on norm-bounded sets. Let $N^+$ be the cone of positive functionals in $N$, then $N^+$ has a base $B$ consisting of states in $N$.
By the proof of \cite[Theorem 6]{Ellis}, $N$ is base normed with respect to $B$.
By Theorem \ref{th:kadison}, $(V,V^+,u)$ is order isomorphic to $(A_b(B),A_b(B)^+,1_B)$ with pointwise ordering, which implies that $B$ is an ordering set of completely additive states on $E\simeq V[0,u]\simeq A_b(B)[0,1_B]$.

For the converse, assume that $(V,V^+,u)$ is an order unit space which is a Banach space dual. By the Banach-Alaoglu theorem, the unit ball in a dual of a Banach space is $w^*$-compact. In our case, this implies that $[-u,u]$ is $w^*$-compact and using homogeneity and translation invariance, we obtain that $[0,u]=\frac{1}{2} ([-u,u] + u)$ is $w^*$-compact as well. Let $S\subset S(E)$ be the set of (restrictions of) normal states, then $S$ is separating. The topology $\sigma(E,S)$ is coarser than the $w^*$-topology, hence $E$ is compact in  $\sigma(E,S)$.

\end{proof}

To summarize, we obtain the following statement.

\begin{corollary}\label{coro:summary} Let $A$ be a synaptic algebra. The following statements are equivalent:

(i) $A$ is a dual of a Banach space.

(ii) $A$ is monotone complete and has a separating set of normal states.

(i) $E$ is orthocomplete and has a separating set of completely additive states.

(iii) $E$ is compact in the topology defined by a separating set of states.

(iv) $A$ is a JW-algebra.

(v) $A$ is the dual of a base norm space whose base is the set of normal states of $A$.
\end{corollary}

The proof of the next theorem follows from Theorem \ref{th:smearing}. We give here a more direct proof based on Theorem \ref    {th:kadison}.

\begin{theorem}\label{th:JWsmearing} Let $A$ be a JW-algebra and $(X,{\mathcal A})$, $(Y,{\mathcal B})$ be measurable spaces.
Let $\xi$ be an $(X,\mathcal A)$-observable. For  every weak Markov kernel $\nu:X\times {\mathcal B} \to [0,1]$ with respect to $I_\xi$, there exists an observable $\eta$ on $E$ which is the smearing of $\xi$.
\end{theorem}

\begin{proof} Let $K$ be the set of normal states on $A$. By Corollary \ref{coro:summary} and Theorem \ref{th:kadison}, $E$ is isomorphic to the set of all affine functions from $K$ to ${\mathbb R}[0,1]$. For every $s\in K$,
$s\circ \xi$ is a probability measure on $(X,{\mathcal A})$, and for every $B\in {\mathcal B}$, $\nu_B: X \to {\mathbb R}[0,1]$  is a measurable function on $X$. Put $\eta(B)(s):=s\circ\xi(\nu_B)=\int_X \nu_B(x)s\circ \xi(dx)$. Then $s\mapsto \eta(B)(s)$ is an affine function on $K$ with values in ${\mathbb R}[0,1]$,
therefore $\eta(B)\in E$. To prove that $B\mapsto \eta(B)$ is an observable, let $(B_i)_{i=1}^{\infty}$
be a sequence of pairwise disjoint sets with $\cup_i B_i=B$. By the properties of a weak Markov kernel we have $\nu(x,B)=\sum_i\nu(x,B_i)$ $\xi$-a.e. For every $s\in K$ we have
\begin{align*}
\eta(B)(s)&=\int_X\nu_B(x)s\circ \xi(dx)=
\int_X\sum_i\nu_{B_i}(x)s\circ \xi(dx)=\sum_i\int_X \nu_{B_i}(x)s\circ \xi(dx)\\
&=\sum_i\eta(B_i)(s),
\end{align*}
 hence $\eta(B)=\sum_i\eta(B_i)$. It then follows that $\eta$ is an observable defined by a smearing of $\xi$. Uniqueness is clear.
\end{proof}

In the next theorem we prove that in case that $A$ is a JW-algebra, also the converse of Theorem \ref{th:commutrange} holds true.

\begin{theorem}Let $A$ be a JW-algebra of operators on a separable Hilbert space. Then an observable $\eta$ is a smearing of a sharp observable $\xi$ if and only if the range ${\mathcal R}(\eta)$ consists of pairwise commuting effects.
\end{theorem}

\begin{proof}  If the range of $\eta$ is pairwise commuting, the result follows by Theorem \ref{th:commutrange}.

For the converse, let $\eta: (Y,{\mathcal B})\to E$ be an observable on $E$ that is a smearing of a sharp observable $\xi: (X,{\mathcal A}) \to P$. This means that for every $\sigma$-additive  state $s\in S(E)$ and every set $B\in {\mathcal B}$, $s(\eta(B))=\int_X\nu(x,B) s(\xi(dx))=s(\nu_B(\xi))$, where $\nu_B(x)=\nu(B,x):{\mathcal B}\times X\to [0,1]$ is a weak Markov kernel, and $\nu_B(\xi)$  is a sharp real observable such that for every $\Delta \in {\mathcal B}({\mathbb R})$, $\nu_B(\xi(\Delta)))=\xi(\nu_B^{-1}(\Delta))\in {\mathcal R}(\xi)$. Since the range ${\mathcal R}(\xi)$ consists of mutually commuting projections, it is contained in a commutative sub-GH-algebra $A_0$. By Theorem \ref{thm:sharp}, there exists an element of $A_0$ corresponding to $\nu_B(\xi)$, and since the normal states are ordering, this element must be equal to $\eta(B)$. It follows that the range of $\eta$ is contained in $A_0$, hence is commutative.
\end{proof}

By \cite[Theorem 4.4]{JPV1}, on the effect algebra ${\mathcal E}(H)$ on the separable Hilbert space $H$, an observable (POV measure) is a smearing of a sharp  real observable (PV-measure) iff it has a commutative range. We prove that the same holds for a JW-algebra of operators on a separable Hilbert space.

\begin{corollary}
Let $A$ be a JW-algebra of operators on a separable Hilbert space.
Let $\eta$ be an observable on $A$ whose range ${\mathcal R}(\eta)$ consist of pairwise commuting effects in $E$. Then $\eta$ is a smearing of a sharp real observable.
\end{corollary}

\begin{proof}
For every $\eta(B), B\in {\mathcal B}$,  there is a sharp real observable $\xi_B$ such that $s(\eta(B))=s(\xi_B)$ for every normal state $s$ on $A$ and every $B\in {\mathcal B}$.
So we have a system $\{ \xi_B:B\in {\mathcal B}\}$ of compatible observables on the OML $P$. By \cite[Theorem 3.9]{Var},  there exists an  observable $\xi$ and real valued Borel measurable functions $f_B:X\to {\mathbb R}$ for all $B\in {\mathcal B}$ such that $\xi_B=f_B\circ \xi$.
Since the OML of projections on a separable Hilbert space is separable (in the sense that every Boolean subalgebra of it is countably generated), by \cite[Theorem 3.9]{Var}, there exists a real observable $\xi$.
Put $\nu(B,x):=f_B(x)$. Similarly as in the proof of \cite[Theorem 4.4]{JPV1}, we prove that $\nu(B,x)$ is a weak Markov kernel and that $\eta$ is a smearing of $\xi$.
\end{proof}

\end{document}